\documentclass[oneside,english]{amsart}
\usepackage[T1]{fontenc}
\usepackage[latin9]{inputenc}
\usepackage{amsthm}
\usepackage{amstext}
\usepackage{amssymb}
\usepackage{esint}
\usepackage{amscd}
\usepackage{color}
\makeatletter
\numberwithin{equation}{section}
\numberwithin{figure}{section}
\theoremstyle{plain}
\newtheorem{thm}{\protect\theoremname}[section]
\newtheorem{prop}[thm]{\protect\propositionname}
\theoremstyle{plain}
\theoremstyle{definition}

\theoremstyle{plain}
\newtheorem{lem}[thm]{\protect\lemmaname}

\theoremstyle{plain}
\newtheorem{rem}[thm]{\protect\remarkname}
\theoremstyle{plain}
\makeatother

\usepackage{babel}
\providecommand{\definitionname}{Definition}
\providecommand{\lemmaname}{Lemma}
\providecommand{\theoremname}{Theorem}
\providecommand{\corollaryname}{Corollary}
\providecommand{\remarkname}{Remark}
\providecommand{\propositionname}{Proposition}
	
\DeclareMathOperator{\loc}{loc}
\DeclareMathOperator{\dist}{dist}

\DeclareMathOperator{\ess}{ess}
\DeclareMathOperator{\diam}{diam}

\DeclareMathOperator{\cp}{cap}

\begin{document}

\title[Capacity of rings and mappings generate embeddings of Sobolev spaces]{Capacity of rings and mappings generate embeddings of Sobolev spaces}

\author{Alexander Menovschikov and Alexander Ukhlov}
\begin{abstract}
In this paper we give characterizations of mappings generate embeddings of Sobolev spaces in the terms of ring capacity inequalities. In addition we prove that such mappings are Lipschitz mappings in the sub-hyperbolic type capacitory metrics.
\end{abstract}
\maketitle
\footnotetext{The first named author was supported by the grant GA\v{C}R 20-19018Y.}

\section{Introduction }

In the Sobolev embedding theory (see, for example, \cite{M}) the significant part have mappings which generate embeddings of Sobolev spaces by means of composition operators \cite{GGu,GU,GU16,GU17}. The geometric theory of composition operators on Sobolev spaces \cite{U93,VU04} represents a generalization of the quasiconformal mappings theory and is closely connected with various generalizations of quasiconformal mappings defined in capacity (moduli) terms, see, for example, \cite{HK14,K86,Kr64,MRSY09,P69,R96,S85}. The mappings generate bounded composition operators on Sobolev spaces are called as weak $(p,q)$-quasiconformal mappings \cite{GGR95,VU98,VU04}. 
The weak $(p,q)$-quasiconformal mappings can be defined in the terms of capacity inequalities \cite{U93,VU98} by using extremal capacity functions \cite{VGR}. In the present work we consider characterizations of the weak $(p,q)$-quasiconformal mappings, $n-1<q\leq p<\infty$, in the terms of ring capacity inequalities. Note that characterizations of (weak) quasiconformal mappings in the terms of ring capacity inequalities arise to the work \cite{Ger62} and have a significant role in the geometric function theory (see, for example, \cite{MRSY09}).

Recall that a homeomorphism $\varphi:\Omega\to\widetilde{\Omega}$ is called a weak $p$-quasiconformal mapping \cite{GGR95} if $\varphi\in W^1_{p,\loc}(\Omega)$, has finite distortion and
$$
K_{p}(\varphi;\Omega)=\ess\sup\limits_{x\in\Omega}\left(\frac{|D\varphi(x)|^p}{|J(x,\varphi)|}\right)^{\frac{1}{p}}<\infty.
$$
In the case $p=q=n$ we obtain the class of quasiconformal mappings. The first main result of the article states: {\it Let $n-1<p<\infty$, then $\varphi :\Omega\to \widetilde{\Omega}$ is a weak $p$-quasiconformal mapping if and only if there exists a constant $C_{p}(\varphi;\Omega)<\infty$ such that for every ring condenser $(F,G)\subset\widetilde{\Omega}$ the inequality
$$
\cp_{p}^{1/p}(\varphi^{-1}(F);\varphi^{-1}(G))
\leq C_{p}(\varphi;\Omega)\cp_{p}^{1/p}(F;G)
$$
holds. }

Weak $(p,q)$-quasiconformal mappings are defined as mappings with integrable $p$-distortion. Recall that a homeomorphism $\varphi:\Omega\to\widetilde{\Omega}$ is called a weak $(p,q)$-quasiconformal mapping \cite{VU98} if $\varphi\in W^1_{q,\loc}(\Omega)$, has finite distortion and
$$
K_{p,q}(\varphi;\Omega)=\left(\int\limits_{\Omega}\left(\frac{|D\varphi(x)|^p}{|J(x,\varphi)|}\right)^{\frac{q}{p-q}}\right)^{\frac{p-q}{pq}}<\infty.
$$

The second main result of the article gives the characterization of a general class of weak $(p,q)$-quasiconformal mappings, $n-1<q<p<\infty$, by means of capacity inequalities for ring condensers: {\it Let $n-1	<q<p<\infty$, then a homeomorphism $\varphi :\Omega\to \widetilde{\Omega}$ is a weak $(p,q)$-quasiconformal mapping if and only if
there exists a bounded monotone countable-additive set function
$\widetilde{\Phi}_{p,q}$ defined on open subsets of $\widetilde{\Omega}$
such that for every ring condenser 
$(F,G)\subset \widetilde{\Omega}$
the inequality
\begin{equation*}
\cp_q^{1/q}(\varphi^{-1}(F);\varphi^{-1}(G)) \leq \widetilde\Phi_{p,q}(G)^{\frac{p-q}{pq}} \cp_p^{1/p}(F;G)
\end{equation*}
holds.
} 

Set functions in the composition operators theory on Sobolev spaces were introduced in \cite{U93} in the framework of the solution of Reshetnyak's problem (1969). In the second part of this article we study capacitary characterizations of the set functions associated with norms of composition operators \cite{VU04,VU05}.

In the end of the introduction, let us recall that quasiconformal mappings are bi-Lipschitz mappings in metrics associated with conformal capacity \cite{FMV91,VGR}. In the present article we define the sub-hyperbolic type $p$-capacitary metrics 
$$
d_p(x,y) = \inf\limits_{\gamma} \cp^{\frac{1}{p}}_p(\gamma; \Omega),\,\, n-1<p\leq n,
$$
where the infimum is taken over all continuous curves $\gamma$, joining points $x$ and $y$ in $\Omega$. 
In the case $p=n$ metrics of such type were considered in \cite{FMV91} in a connection with quasiconformal mappings.

We prove that if $\varphi: \Omega \to \widetilde\Omega$ is a weak $(p,q)$-quasiconformal mapping, $n-1< q \leq p \leq n$, then, for any two points $x,y \in \widetilde\Omega$ the following inequality 
$$
d_q(\varphi^{-1}(x), \varphi^{-1}(y)) \leq K_{p,q}(\varphi;\Omega) d_p(x,y),
$$
holds. In the case $p=q=n$ we have the well known property of quasiconformal mappings \cite{FMV91,GP76}.

\section{Composition operators on Sobolev spaces}

\subsection{Sobolev spaces}

Let us recall the basic notions of the Sobolev spaces.
Let $\Omega$ be an open subset of $\mathbb R^n$. The Sobolev space $W^1_p(\Omega)$, $1\leq p\leq\infty$, is defined \cite{M}
as a Banach space of locally integrable weakly differentiable functions
$f:\Omega\to\mathbb{R}$ equipped with the following norm: 
\[
\|f\mid W^1_p(\Omega)\|=\| f\mid L_p(\Omega)\|+\|\nabla f\mid L_p(\Omega)\|,
\]
where $\nabla f$ is the weak gradient of the function $f$, i.~e. $ \nabla f = (\frac{\partial f}{\partial x_1},...,\frac{\partial f}{\partial x_n})$.

The homogeneous seminormed Sobolev space $L^1_p(\Omega)$, $1\leq p\leq\infty$, is defined as a space
of locally integrable weakly differentiable functions $f:\Omega\to\mathbb{R}$ equipped
with the following seminorm: 
\[
\|f\mid L^1_p(\Omega)\|=\|\nabla f\mid L_p(\Omega)\|.
\]

In the Sobolev spaces theory, a crucial role is played by capacity as an outer measure associated with Sobolev spaces \cite{M}. In accordance to this approach, elements of Sobolev spaces $W^1_p(\Omega)$ are equivalence classes up to a set of $p$-capacity zero \cite{MH72}. 

The mapping $\varphi:\Omega\to\mathbb{R}^{n}$ belongs to the Sobolev space $W^1_{p,\loc}(\Omega,\mathbb R^n)$, if its coordinate functions belong to $W^1_{p,\loc}(\Omega)$. In this case, the formal Jacobi matrix $D\varphi(x)$ and its determinant (Jacobian) $J(x,\varphi)$
are well defined at almost all points $x\in\Omega$. The norm $|D\varphi(x)|$ is the operator norm of $D\varphi(x)$.

\subsection{Composition operators}

Let $\Omega$ and $\widetilde{\Omega}$ be domains in the Euclidean space $\mathbb R^n$. Then a homeomorphism $\varphi:\Omega\to\widetilde{\Omega}$ generates a bounded composition
operator 
\[
\varphi^{\ast}:L^1_p(\widetilde{\Omega})\to L^1_q(\Omega),\,\,\,1\leq q\leq p\leq\infty,
\]
by the composition rule $\varphi^{\ast}(f)=f\circ\varphi$, if for
any function $f\in L^1_p(\widetilde{\Omega})$, the composition $\varphi^{\ast}(f)\in L^1_q(\Omega)$
is defined quasi-everywhere in $\Omega$ and there exists a constant $K_{p,q}(\varphi;\Omega)<\infty$ such that 
\[
\|\varphi^{\ast}(f)\mid L^1_q(\Omega)\|\leq K_{p,q}(\varphi;\Omega)\|f\mid L^1_p(\widetilde{\Omega})\|.
\]

Recall that the $p$-dilatation \cite{Ger69} of a Sobolev mapping $\varphi: \Omega\to \widetilde{\Omega}$ at the point $x\in\Omega$ is defined as
$$
K_p(x)=\inf \{k(x): |D\varphi(x)|\leq k(x) |J(x,\varphi)|^{\frac{1}{p}}\}.
$$

The following theorem gives the characterization of composition operators in terms of integral characteristics of mappings of finite distortion. Homeomorphisms $\varphi:\Omega\to\widetilde{\Omega}$, which satisfy conditions of this theorem, are weak $(p,q)$-quasiconformal mappings.

Recall that a weakly differentiable mapping $\varphi:\Omega\to\mathbb{R}^{n}$ is a mapping of finite distortion if $D\varphi(x)=0$ for almost all $x$ from $Z=\{x\in\Omega: J(x,\varphi)=0\}$ \cite{VGR}. 

\begin{thm}
\label{CompTh} Let $\varphi:\Omega\to\widetilde{\Omega}$ be a homeomorphism
between two domains $\Omega$ and $\widetilde{\Omega}$. Then $\varphi$ generates a bounded composition
operator 
\[
\varphi^{\ast}:L^1_p(\widetilde{\Omega})\to L^1_{q}(\Omega),\,\,\,1\leq q\leq p\leq\infty,
\]
 if and only if $\varphi\in W^1_{1,\loc}(\Omega)$, has finite distortion if $p<\infty$,
and 
\[
K_{p,q}(\varphi;\Omega) := \|K_p \mid L_{\kappa}(\Omega)\|<\infty, \,\,1/q-1/p=1/{\kappa}\,\,(\kappa=\infty, \text{ if } p=q).
\]
The norm of the operator $\varphi^\ast$ is estimated as $\|\varphi^\ast\| \leq K_{p,q}(\varphi;\Omega)$.
\end{thm}

This theorem in the case $p=q=n$ was given in the work \cite{VG75}. The general case $1\leq q\leq p<\infty$ was proved in \cite{U93}, where the weak change of variables formula \cite{H93} was used (see, also the case $n<q=p<\infty$ in \cite{V88}). The limit case $p=\infty$ was considered in \cite{GU10}. The geometric characteristics of mappings which generate bounded composition operators in the case $n-1<q=p<\infty$ were given in \cite{GGR95}. 

It is known (see, for example, \cite{V71}) that mappings which are inverse to quasiconformal homeomorphisms are quasiconformal also. In the case of weak $(p,q)$-quasiconformal mappings with $n-1<q\leq p<\infty$, the following composition regularity theorem was given in \cite{U93}.

\begin{thm}
\label{CompThD} Let a homeomorphism $\varphi:\Omega\to\widetilde{\Omega}$
between two domains $\Omega$ and $\widetilde{\Omega}$ generate a bounded composition
operator 
\[
\varphi^{\ast}:L^1_p(\widetilde{\Omega})\to L^1_{q}(\Omega),\,\,\,n-1<q \leq p< \infty.
\]
Then the inverse mapping $\varphi^{-1}:\widetilde{\Omega}\to\Omega$ generates a bounded composition operator 
\[
\left(\varphi^{-1}\right)^{\ast}:L^1_{q'}(\Omega)\to L^1_{p'}(\widetilde{\Omega}),
\]
where $p'=p/(p-n+1)$, $q'=q/(q-n+1)$. 
\end{thm}

\section{Capacity inequalities for ring condensers}

In the quasiconformal mappings theory, distortion properties of ring condensers play a significant role. 
The ring condenser $R$ (see, for example, \cite{Ger69,MRV69}) in a domain $\Omega\subset \mathbb R^n$ is a pair $R=(F;G)$ of sets $F\subset G\subset\Omega$, where $F$ is a connected compact set and $G$ is a connected open set.  
The $p$-capacity of the ring $R=(F;G)$ is defined by
$$
\cp_p(F;G) =\inf\|f\vert L_p^1(\Omega)\|^p,
$$
where the greatest lower bound is taken over all continuous functions $f\in L_p^1(\Omega)$ with a compact support in $G$ and such that $f\geq 1$ on $F$. Such functions are called admissible functions for the ring $R=(F;G)$. 

In this section we prove that weak $(p,q)$-quasiconformal mappings, $n-1<q\leq p<\infty$, can be characterized by considering the ring condensers only. Before formulating the main results, we recall the estimates of the capacity of ring condensers.

Let the condenser $R=(F;G)$ consist of two concentric balls $F = \{x: |x| \leq r_F\}$, $G= \{x: |x| < r_G\}$, $0<r_F<r_G$, then (see, e.g. \cite{GResh}) 
\begin{equation}\label{capballs}
  \cp_p^{1/p}(F;G) = C(n,p) (r_F^{\frac{p-n}{p-1}}-r_G^{\frac{p-n}{p-1}})^{\frac{1-p}{p}}.
\end{equation}

In the case of arbitrary $F$ and $G$, the following lower estimate for the capacity of a ring condenser was given in \cite{M}
\begin{equation}\label{caplower}
   \cp_p^{1/p}(F; G) \geq \frac{\inf m_{n-1}S}{|G|^{1-\frac{1}{p}}},
\end{equation}
where $S$ is a $C^\infty$-manifold, which is a boundary of an open set $U \supset F$, $\overline U \subset G$, $m_{n-1} S$ its $(n-1)$-Lebesgue measure, and infimum is taken over all such manifolds.

If we additionally assume that $F$ is a connected set, then the following estimate holds (see \cite{K86, VU98}):
\begin{equation}\label{caplowerdiam}
   \cp_p^{1/p}(F; G) \geq C(n,p) \frac{(\diam(F))^{\frac{1}{n-1}}}{|G|^{\frac{1}{n-1}-\frac{1}{p}}}.
\end{equation}

We need also an upper estimate of the capacity of the ring condenser (see, e.g. \cite{MRV69, K86}),
\begin{equation}\label{capupper}
   \cp_p^{1/p}(F;G) \leq \frac{|G \setminus F|^{\frac{1}{p}}}{\dist(F, \partial G)}
\end{equation}

Now we formulate the characterization of weak $p$-quasiconformal mappings in terms of the distortion of ring condensers, that generalized results of \cite{Ger62}. 

\begin{thm}
\label{cap-pp}
\it Let $n-1<p<\infty$, then $\varphi :\Omega\to \widetilde{\Omega}$ is a weak $p$-quasiconformal mapping if and only if there exists a constant $C_{p}(\varphi;\Omega)>0$ such that for every ring condenser $(F,G)\subset\widetilde{\Omega}$ the inequality
\begin{equation}
\label{ringca}
\cp_{p}^{1/p}(\varphi^{-1}(F);\varphi^{-1}(G))
\leq C_{p}(\varphi;\Omega)\cp_{p}^{1/p}(F;G)
\end{equation}
holds.
\end{thm}

\begin{proof} In the case $p=n$ we have the well known property of quasiconformal mappings \cite{Ger62,V71}. Hence we will prove the theorem in the case $p\ne n$.

\vskip 0.2cm
\noindent
{\it Necessity.} Let $\varphi :\Omega\to \widetilde{\Omega}$ be a weak $p$-quasiconformal mapping. Then by Theorem~\ref{CompTh} the inequality 
$$
\|\varphi^{\ast}(f)\mid L^1_p(\Omega)\|\leq K_{p}(\varphi;\Omega) \|f\mid L^1_p(\widetilde{\Omega})\|
$$
holds for every function $f\in L^1_p(\widetilde{\Omega})$. Let $f$ be an admissible function for the ring condenser $R=(F;G)$, then $f\circ\varphi$ is an admissible function for the ring condenser $\varphi^{-1}(R)=(\varphi^{-1}(F);\varphi^{-1}(G))$. Hence, by the definition of the capacity we have 
$$
\cp_p^{1/p} (\varphi^{-1}(F);\varphi^{-1}(G))\leq 
\|\varphi^{\ast}(f)\mid L^1_p(\Omega)\|\leq K_{p}(\varphi;\Omega) \|f\mid L^1_p(\widetilde{\Omega})\|.
$$
Since $f$ is arbitrary admissible function, we obtain
$$
\cp_p^{1/p} (\varphi^{-1}(F);\varphi^{-1}(G))\leq K_{p}(\varphi;\Omega) \cp_{p}^{1/p}(F;G).
$$
with $C_{p}(\varphi;\Omega) = K_{p}(\varphi;\Omega)$

\vskip 0.2cm
\noindent
{\it Sufficiency.}
By Theorem~\ref{CompTh}, it is sufficient to prove that $\varphi \in W^1_{1,\loc}(\Omega)$, has finite distortion and 
$$
K_{p}(\varphi;\Omega) = \ess\sup\limits_{x\in\Omega}\frac{|D\varphi(x)|}{|J(x,\varphi)|^{\frac{1}{p}}}<\infty
$$
On the first step we prove that $\varphi^{-1}$ satisfies the conditions of Theorem~\ref{CompTh} and generates a bounded composition operator
$$
(\varphi^{-1})^\ast : L^1_{p'}(\Omega) \to L^1_{p'}(\widetilde\Omega), \quad n-1 < p'<\infty,
$$
where $p' = \frac{p}{p-(n-1)}$.

In \cite{Ger69} it was proved that in the case $n-1<p<n$ the capacity inequality~(\ref{ringca}) implies $\varphi^{-1}\in W^1_{\infty,\loc}(\widetilde\Omega)$. In the case $p>n$ by using methods of \cite{U93}, it follows that the capacity inequality 
$$
\cp_p^{1/p}(\varphi^{-1}(F);\varphi^{-1}(G)) \leq C_{p}(\varphi;\Omega) \cp_p^{1/p}(F;G)
$$
implies $\varphi^{-1}\in W^1_{1,\loc}(\widetilde\Omega)$.

Now, for every $y \in \widetilde\Omega$, we consider a ring condenser $(F_r; G_r)$: 
$$
F_r = \{z : |y-z| \leq r\}, \,\,G_r = \{z : |y-z| < 2r\},
$$
where small enough $r>0$ is chosen such that $(F_r; G_r) \subset \widetilde\Omega$. Then, by \eqref{capballs},
$$
  \cp_p^{1/p}(F_r; G_r) = C(p,n) r^{\frac{n-p}{p}}.
$$
By \eqref{caplowerdiam}, the condenser $(\varphi^{-1}(F_r); \varphi^{-1}(G_r))$ can be estimates by
$$
  \cp_p^{1/p}(\varphi^{-1}(F_r); \varphi^{-1}(G_r)) \geq C(n,p) \frac{(\diam(\varphi^{-1}(F_r)))^{\frac{1}{n-1}}}{|\varphi^{-1}(G_r)|^{\frac{1}{n-1}-\frac{1}{p}}}.
$$

Substituting these estimates in the inequality~\eqref{ringca} we obtain
$$
(\diam(\varphi^{-1}(F_r)))^{\frac{1}{n-1}} \leq C(n,p) C_{p}(\varphi;\Omega) r^{\frac{n-p}{p}} |\varphi^{-1}(G_r)|^{\frac{1}{n-1}-\frac{1}{p}},
$$
or, in another form,
$$
\left(\frac{\diam(\varphi^{-1}(F_r))}{r}\right)^{\frac{1}{n-1}} \leq C(n,p) C_{p}(\varphi;\Omega) \left(\frac{|\varphi^{-1}(G_r)|}{|G_r|}\right)^{\frac{1}{n-1}-\frac{1}{p}}.
$$
Taking $r \to 0$, we obtain 
$$
(L(y, \varphi^{-1}))^{\frac{1}{n-1}} \leq C(n,p) C_{p}(\varphi;\Omega) |J(y, \varphi^{-1})|^{\frac{p-(n-1)}{p(n-1)}},
$$
for almost all $y \in \widetilde\Omega$,
where 
$$
L(y, \varphi^{-1}) = \limsup\limits_{r \to 0}\left(\max\limits_{|y-z|=r}\frac{|\varphi^{-1}(y)-\varphi^{-1}(z)|}{r}\right),
$$ 
and the volume derivative 
$$
|J(y, \varphi^{-1})|=\lim\limits_{r \to 0}\frac{|\varphi^{-1}(B(y,r)))|}{|B(y,r)|},
$$
are finite for almost all $y \in \widetilde\Omega$. Hence by Stepanov's Theorem the inverse mapping $\varphi^{-1}$ is differentiable a.e. in $\widetilde\Omega$.

Since $|D\varphi^{-1}(y)|\leq L(y, \varphi^{-1})$ for almost all $y \in \widetilde\Omega$, raising the above inequality to the power $\frac{p(n-1)}{p-(n-1)}=p'(n-1)$, we obtain
$$
|D\varphi^{-1}(y)|^{p'} \leq C(n,p) (C_{p}(\varphi;\Omega))^{p'(n-1)} |J(y, \varphi^{-1})| \text{ a.e. in } \widetilde\Omega.
$$
This pointwise inequality implies that $\varphi^{-1}$ is a mapping of finite distortion. 

It also means that
$$
K_{p}(\varphi^{-1};\widetilde\Omega) = \ess\sup\limits_{y\in\widetilde\Omega}
\left(\frac{|D\varphi^{-1}(y)|^{p'}}{|J(y, \varphi^{-1})|}\right)^{\frac{1}{p'}} \leq C(n,p)(C_{p}(\varphi;\Omega))^{(n-1)} < \infty,
$$
and, therefore, by~Theorem \ref{CompTh} $\varphi^{-1}$ generates a bounded composition operator
 $$
   (\varphi^{-1})^\ast : L^1_{p'}(\Omega) \to L^1_{p'}(\widetilde\Omega), \quad n-1 < p'<\infty
 $$
 
Hence, by Theorem \ref{CompThD}, the mapping $\varphi$ generates a bounded composition operator
$$
\varphi^\ast : L^1_{p''}(\widetilde\Omega) \to L^1_{p''}(\Omega), \quad 1 < p''<\infty.
$$
So $\varphi \in W^1_{1,\loc}(\Omega)$ and is the mapping of finite distortion. 

Now we prove that
 $$
   K_{p}(\varphi;\Omega) = \ess\sup\limits_{x\in\Omega}\left(\frac{|D\varphi(x)|^p}{|J(x,\varphi)|}\right)^{\frac{1}{p}}= \ess\sup\limits_{x\in\Omega}\frac{|D\varphi(x)|}{|J(x,\varphi)|^{\frac{1}{p}}} \leq C_{p}(\varphi;\Omega).
 $$

Since
$$
  \ess\sup\limits_{x\in\Omega}\frac{|D\varphi(x)|}{|J(x,\varphi)|^{\frac{1}{p}}} \leq \ess\sup\limits_{y\in\widetilde\Omega}\frac{|J(y,\varphi^{-1})|^{\frac{1}{p}}}{(l(D\varphi^{-1}(y)))}
$$
it is sufficient to prove the following estimate for the inverse mapping:
 \begin{equation}\label{estinv-pp}
   \ess\sup\limits_{y\in\widetilde\Omega}\frac{|J(y,\varphi^{-1})|^{\frac{1}{p}}}{(l(D\varphi^{-1}(y)))} \leq C_{p}(\varphi;\Omega).
 \end{equation}
 
To prove \eqref{estinv-pp}, we consider a condenser $\hat R = (\hat F; \hat G)$ of the following specific form. We will also use this condenser in the proof of the next theorem in the case of weak $(p,q)$-quasiconformal mappings. 

Since $\varphi^{-1}$ is differentiable almost everywhere in $\widetilde{\Omega}$ we consider an point $y_0 \in \widetilde\Omega$, where $\varphi^{-1}$ is differentiable.
Denote as $\lambda_1 \geq \lambda_2 \geq \dots \geq \lambda_n > 0$ semiaxes of an ellipsoid, which is an image of the unit ball in $\widetilde\Omega$ under the linear mapping $D\varphi^{-1}(y_0)$. Without loss of generality, we can assume that $y_0 = 0$, $\varphi^{-1}(0) = 0$ and $\left(D\varphi^{-1}(0)\right)(e_i) = \lambda_i e_i$, $i = 1, \dots, n$.  Fix $t>0$ and choose $r>0$ such that the condenser $\hat{R}$ belongs to $\widetilde\Omega$, where 
 \begin{equation}\label{condbox}
 \begin{split}
 &\hat{R}=(\hat{F};\hat{G}), \\
 &\hat{F} = \{y : y_n=0, |y_i|\leq r, i=1,\dots,n-1\}, \\
 &\hat{G} = \{y : |y_n|\leq rt\lambda_n, |y_i|<r + rt\lambda_i, i=1,\dots,n-1\}. 
 \end{split}
 \end{equation}
 
 For this condenser, by \eqref{capupper}
 $$
   \cp_p^{1/p}(F;G) \leq \frac{|G \setminus F|^{\frac{1}{p}}}{\dist(F, \partial G)} \leq \frac{2^{\frac{n}{p}} r^{\frac{n}{p}} \lambda_n^{\frac{1}{p}}t^{\frac{1}{p}}\prod\limits_{i=1}\limits^{n-1} (1+t\lambda_i)^{\frac{1}{p}}}{rt\lambda_n} = \frac{2^{\frac{n}{p}}\prod\limits_{i=1}\limits^{n-1} (1+t\lambda_i)^{\frac{1}{p}}}{\lambda_n^{1-\frac{1}{p}}t^{1-\frac{1}{p}}r^{1-\frac{n}{p}}} 
 $$
 
For the application of the  lower estimate \eqref{caplower} of the condenser $(\varphi^{-1}(\hat{F}); \varphi^{-1}(\hat{G}))$, we  estimate $|\varphi^{-1}(G)|$ and $m_{n-1}S$, where $S$ is the $C^\infty$-manifold from \eqref{caplower}.
Since $\varphi^{-1}$ is differentiable at the point $y_0=0$, we can fix an arbitrary $\varepsilon$, $0< \varepsilon \lambda_n$ and choose $r>0$ such that $|\varphi^{-1}(y) - (D\varphi^{-1}(0))(y)| < \varepsilon r$ for $y\in G$. Then, $\varphi^{-1}(G)$ is a subset of a parallelepiped
 $$
   P_G = \{x : |x_n| \leq rt\lambda^2_n+\varepsilon r, |x_i|\leq r\lambda_i + rt\lambda^2_i+ \varepsilon r, i=1,\dots,n-1\},
 $$
 and a projection of $\varphi^{-1}(F)$ on $x_n=0$ contain as a subset $(n-1)$-dimensional parallelepiped
 $$
   P_F = \{x : x_n=0, |x_i|\leq r\lambda_i-\varepsilon r, i=1,\dots,n-1\}.
 $$
 Therefore,
 \begin{equation}\label{estboxlow}
 \begin{split}
 & |\varphi^{-1}(G)| \leq |P_G|= 2^nr^n(t\lambda^2_n+\varepsilon)\prod\limits_{i=1}\limits^{n-1}(\lambda_i+t\lambda^2_i+\varepsilon),\\
 & m_{n-1}S \geq 2m_{n-1}P_F = 2^nr^{n-1}\prod\limits_{i=1}\limits^{n-1}(\lambda_i-\varepsilon).
 \end{split}
 \end{equation} 

Because of   
$|J(y_0,\varphi^{-1})| = \prod\limits_{i=1}\limits^{n}\lambda_i$, $l(D\varphi^{-1}(y_0)) = \lambda_n$, the inequality \eqref{estinv-pp} is equal to
 $$
  \frac{\prod\limits_{i=1}\limits^{n}\lambda_i^{\frac{1}{p}}}{\lambda_n} \leq C_{p}(\varphi;\Omega).
 $$
 
Combine the estimates for the condenser $\hat{R}=(\hat{F};\hat{G})$ we obtain from the inequality  \eqref{ringca} 
 $$
 \frac{2^{\frac{n}{p}}r^{\frac{n}{p}-1}\prod\limits_{i=1}\limits^{n-1}(\lambda_i-\varepsilon)}{(t\lambda^2_n+\varepsilon)^{1-\frac{1}{p}}\prod\limits_{i=1}\limits^{n-1}(\lambda_i+t\lambda^2_i+\varepsilon)^{1-\frac{1}{p}}}
 \leq C(n,p)C_{p}(\varphi;\Omega) \frac{2^{\frac{n}{p}}\prod\limits_{i=1}\limits^{n-1} (1+t\lambda_i)^{\frac{1}{p}}}{\lambda_n^{1-\frac{1}{p}}t^{1-\frac{1}{p}}r^{1-\frac{n}{p}}}.
 $$
 Let $\varepsilon \to 0$, then
 $$
 \frac{\prod\limits_{i=1}\limits^{n-1}\lambda_i}{\lambda_n^{1-\frac{1}{p}}\prod\limits_{i=1}\limits^{n-1}(\lambda_i+t\lambda^2_i)^{1-\frac{1}{p}}}
 \leq C(n,p)C_{p}(\varphi;\Omega) \prod\limits_{i=1}\limits^{n-1} (1+t\lambda_i)^{\frac{1}{p}}.
 $$
 Further, tending $t \to 0$, we obtain
 $$
 \frac{\prod\limits_{i=1}\limits^{n-1}\lambda_i}{\lambda_n^{1-\frac{1}{p}}\prod\limits_{i=1}\limits^{n-1}\lambda_i^{1-\frac{1}{p}}}
 \leq C(n,p)C_{p}(\varphi;\Omega).
 $$
 Finally,
 $$
  \frac{\prod\limits_{i=1}\limits^{n}\lambda_i^{\frac{1}{p}}}{\lambda_n} \leq C(n,p) C_{p}(\varphi;\Omega)
 $$
 and the theorem is proved.

\end{proof}

In the case $n-1<q<p<\infty$, we use the set functions associated with norms of composition operators, which were introduced in \cite{U93}.
Let us recall the notion of the set function $\widetilde{\Phi}_{p,q}(\widetilde A)$, defined on open bounded subsets $\widetilde A\subset\widetilde{\Omega}$ and associated with the composition operator $\varphi^\ast: L^1_p(\widetilde{\Omega})\to L^1_{q}(\Omega)$, $1 \leq q < p < \infty$:

\begin{equation}\label{setfunc}
\widetilde{\Phi}_{p,q}(\widetilde{A})=\sup\limits_{f\in L^1_p(\widetilde{A})\cap C_0(\widetilde{A})}\left(\frac{\|\varphi^{\ast}(f)\mid L^1_q(\Omega)\|}{\|f\mid L^1_p(\widetilde{A})\|}\right)^{\kappa}, \quad 1/{\kappa}=1/q-1/p.
\end{equation}

\begin{thm}\cite{U93}
\label{thmsetfunc}
\label{CompPhi} Let a homeomorphism $\varphi:\Omega\to\widetilde{\Omega}$
between two domains $\Omega$ and $\widetilde{\Omega}$ generates a bounded composition
operator 
\[
\varphi^{\ast}:L^1_p(\widetilde{\Omega})\to L^1_q(\Omega),\,\,\,1\leq q< p\leq\infty.
\]
Then the function $\widetilde{\Phi}_{p,q}(\widetilde{A})$, defined by (\ref{setfunc}), is a bounded monotone countably additive set function defined on open bounded subsets $\widetilde{A}\subset\widetilde{\Omega}$.
\end{thm}

Recall that a nonnegative mapping $\Phi$ defined on open subsets of $\Omega$ is called a monotone countably additive set function \cite{RR55,VU04} if

\noindent
1) $\Phi(U_1)\leq \Phi(U_2)$ if $U_1\subset U_2\subset\Omega$;

\noindent
2)  for any collection $U_i \subset U \subset \Omega$, $i=1,2,...$, of mutually disjoint open sets
$$
\sum_{i=1}^{\infty}\Phi(U_i) = \Phi\left(\bigcup_{i=1}^{\infty}U_i\right).
$$

The following lemma gives properties of monotone countably additive set functions defined on open subsets of $\Omega\subset \mathbb R^n$ \cite{RR55,VU04}.

\begin{lem}
\label{lem:AddFun}
Let $\Phi$ be a monotone countably additive set function defined on open subsets of the domain $\Omega\subset \mathbb R^n$. Then

\noindent
(a) at almost all points $x\in \Omega$ there exists a finite derivative
$$
\lim\limits_{r\to 0}\frac{\Phi(B(x,r))}{|B(x,r)|}=\Phi'(x);
$$

\noindent
(b) $\Phi'(x)$ is a measurable function;

\noindent
(c) for every open set $U\subset \Omega$ the inequality
$$
\int\limits_U\Phi'(x)~dx\leq \Phi(U)
$$
holds.
\end{lem}

The following theorem gives a characterization of weak $(p,q)$-quasiconformal mappings, $n-1<q<p<\infty$ in the terms of ring capacity inequalities.

\begin{thm}
\label{cap-pq}
Let $n-1<q<p<\infty$, then a homeomorphism $\varphi :\Omega\to \widetilde{\Omega}$ is a weak $p$-quasiconformal mapping if and only if
there exists a bounded monotone countable-additive set function
$\widetilde{\Phi}_{p,q}$ defined on open subsets of $\widetilde{\Omega}$
such that for every ring condenser 
$(F,G)\subset \widetilde{\Omega}$
the inequality
\begin{equation}
\label{ringcap}
\cp_q^{1/q}(\varphi^{-1}(F);\varphi^{-1}(G)) \leq \widetilde\Phi_{p,q}(G)^{\frac{p-q}{pq}} \cp_p^{1/p}(F;G)
\end{equation}
holds.
\end{thm}

\begin{proof}
{\it Necessity.} Let $\varphi :\Omega\to \widetilde{\Omega}$ be a weak $(p,q)$-quasiconformal mappings. Then, by Theorem~\ref{CompTh}, the inequality 
$$
\|\varphi^{\ast}(f)\mid L^1_q(\Omega)\|\leq K_{p,q}(\varphi;\Omega) \|f\mid L^1_p(\widetilde{\Omega})\|
$$
holds for every function $f\in L^1_p(\widetilde{\Omega})$. Let $f$ be an admissible function for the ring condenser $R=(F;G)$, then $f\circ\varphi$ is an admissible function for the ring condenser $\varphi^{-1}(R)=(\varphi^{-1}(F);\varphi^{-1}(G))$. Hence, by the definition of the capacity, we have 
$$
\cp_q^{1/q} (\varphi^{-1}(F);\varphi^{-1}(G))\leq 
\|\varphi^{\ast}(f)\mid L^1_q(\Omega)\|\leq K_{p,q}(\varphi;\Omega) \|f\mid L^1_p(\widetilde{\Omega})\|.
$$
Since $f$ is arbitrary admissible function, we obtain
$$
\cp_q^{1/q} (\varphi^{-1}(F);\varphi^{-1}(G))\leq K_{p,q}(\varphi;\Omega) \cp_{p}^{1/p}(F;G).
$$

\vskip 0.2cm
\noindent
{\it Sufficiency.}
By Theorem~\ref{CompTh}, it is sufficient to prove that $\varphi \in W^1_{1,\loc}(\Omega)$, has finite distortion and 
$$
\int\limits_{\Omega}\left(\frac{|D\varphi(x)|^p}{|J(x,\varphi)|}\right)^{\frac{q}{p-q}} \, dx<\infty.
$$
As in proof of the previous theorem, at first we prove that $\varphi^{-1}$ satisfies the conditions of Theorem~\ref{CompTh} and generates a bounded composition operator
$$
(\varphi^{-1})^\ast : L^1_{q'}(\Omega) \to L^1_{p'}(\widetilde\Omega), \quad n-1 < p'<q'<\infty,
$$
where $p' = \frac{p}{p-(n-1)}$ and $p' = \frac{q}{q-(n-1)}$.

By \cite{U93}, it follows that the capacity inequality 
$$
\cp_q^{1/q}(\varphi^{-1}(F);\varphi^{-1}(G)) \leq \widetilde\Phi_{p,q}(G)^{\frac{p-q}{pq}} \cp_p^{1/p}(F;G)
$$
implies $\varphi^{-1}\in W^1_{1,\loc}(\widetilde\Omega)$.
Similar to the case $p=q$, for each $y \in \widetilde\Omega$, consider a ring condenser $(F_r; G_r)$, where $F_r = \{z : |y-z| \leq r\}$, $G_r = \{z : |y-z| < 2r\}$ and let small enough $r>0$ be chosen such that $(F_r; G_r) \subset \widetilde\Omega$. Then, by substituting estimates \eqref{capballs} and \eqref{caplowerdiam} in the inequality~\eqref{ringcap} we obtain
$$
(\diam(\varphi^{-1}(F_r)))^{\frac{1}{n-1}} \leq C(n,p,q) \widetilde\Phi_{p,q}(G_r)^{\frac{p-q}{pq}} r^{\frac{n-p}{p}} |\varphi^{-1}(G_r)|^{\frac{1}{n-1}-\frac{1}{q}},
$$
 or, in another form,
 $$
   \left(\frac{\diam(\varphi^{-1}(F_r))}{r}\right)^{\frac{1}{n-1}} \leq C \left(\frac{\widetilde\Phi_{p,q}(G_r)}{|G_r|}\right)^{\frac{p-q}{pq}} \left(\frac{|\varphi^{-1}(G_r)|}{|G_r|}\right)^{\frac{1}{n-1}-\frac{1}{q}}.
 $$
 Taking $r \to 0$, we obtain 
$$
(L(y, \varphi^{-1}))^{\frac{1}{n-1}} \leq C (\widetilde\Phi'_{p,q}(y))^{\frac{p-q}{pq}} |J(y, \varphi^{-1})|^{\frac{q-(n-1)}{q(n-1)}}<\infty,
$$
for almost all $y \in \widetilde\Omega$. Therefore, by Stepanov's Theorem, $\varphi^{-1}$ is differentiable a.e. in $\widetilde\Omega$.

Raising the above inequality to the power $\frac{p(n-1)}{p-(n-1)}=p'(n-1)$, we obtain
\begin{equation}\label{pointinv}
|D\varphi^{-1}(y)|^{p'} \leq C (\widetilde\Phi'_{p,q}(y))^{\frac{(p-q)(n-1)}{qp-q(n-1)}} |J(y, \varphi^{-1})|^{\frac{pq-p(n-1)}{qp -q(n-1)}} \text{ a.e. in } \widetilde\Omega,
 \end{equation}
and so $\varphi^{-1}$ is a mapping of finite distortion. 

Using $q' = \frac{q}{q-(n-1)}$, we can rewrite \eqref{pointinv} in the following form:
$$
  \left( \frac{|D\varphi^{-1}(y)|^{q'}}{|J(y, \varphi^{-1})|}\right)^{\frac{p'}{q'-p'}} \leq C \widetilde\Phi'_{p,q}(y)\,\, 
	\text{for almost all}\,\,y\in \widetilde\Omega.
$$
Integrating over $\widetilde\Omega$, we conclude that 
$$
\int\limits_{\widetilde\Omega}\left( \frac{|D\varphi^{-1}(y)|^{q'}}{|J(y, \varphi^{-1})|}\right)^{\frac{p'}{q'-p'}}~dy
\leq C \int\limits_{\widetilde\Omega} \widetilde\Phi'_{p,q}(y)~dy= C \widetilde\Phi(\widetilde\Omega)<\infty
$$
and $\varphi^{-1}$ generates a bounded composition operator
 $$
   (\varphi^{-1})^\ast : L^1_{q'}(\Omega) \to L^1_{p'}(\widetilde\Omega), \quad n-1 < p'<q'<\infty
 $$
 
Hence, by Theorem \ref{CompThD}, the mapping $\varphi$ generates a bounded composition operator
$$
\varphi^\ast : L^1_{p''}(\widetilde\Omega) \to L^1_{q''}(\Omega), \quad 1 < p''<q''<\infty.
$$
So $\varphi \in W^1_{1,\loc}(\Omega)$ and has finite distortion.

Now we prove that
 $$
   \int_\Omega K_p(x)^{\frac{pq}{p-q}} \, dx \leq \widetilde\Phi_{p,q}(\widetilde\Omega).
 $$
By the change of variables formula \cite{H93}
$$
  \int\limits_{\Omega}\left(\frac{|D\varphi(x)|^p}{|J(x,\varphi)|}\right)^{\frac{q}{p-q}} \, dx \leq \int\limits_{\widetilde\Omega} \left(\frac{|J(y,\varphi^{-1})|}{(l(D\varphi^{-1}(y)))^q}\right)^{\frac{p}{p-q}} \, dy
$$
 and so it is sufficient to prove the pointwise estimate
 \begin{equation}\label{pointest}
   \left(\frac{|J(y,\varphi^{-1})|}{(l(D\varphi^{-1}(y)))^q}\right)^{\frac{p}{p-q}} \leq \widetilde\Phi'_{p,q}(y)
 \end{equation}
 for almost all $y \in \widetilde\Omega$.
 
Consider an arbitrary point $y_0 \in \widetilde\Omega$, where $\varphi^{-1}$ is differentiable and $\widetilde\Phi'_{p,q}(y_0)<\infty$ and denote as $\lambda_1 \geq \lambda_2 \geq \dots \geq \lambda_n > 0$ semiaxes of an ellipsoid, which is an image of the unit ball in $\widetilde\Omega$ under the linear mapping $D\varphi^{-1}(y_0)$. Then the inequality \eqref{pointest} is equal to
$$
\frac{\prod\limits_{i=1}\limits^{n}\lambda_i^{\frac{1}{q}}}{\lambda_n} \leq (\widetilde\Phi'_{p,q}(y_0))^{\frac{p-q}{pq}}.
$$

Without loss of generality, we can assume that $y_0 = 0$, $\varphi^{-1}(0) = 0$ and 
$$
(D\varphi^{-1}(0))(e_i) = \lambda_i e_i,\,\, i = 1, \dots, n,
$$ 
Consider the condenser $\hat{R}=(\hat{F};\hat{G})$ defined by \eqref{condbox}. Then, by the estimate \eqref{capupper}
 $$
   \cp_p^{1/p}(F;G) \leq \frac{|G \setminus F|^{\frac{1}{p}}}{\dist(F, \partial G)} \leq \frac{2^{\frac{n}{p}}\prod\limits_{i=1}\limits^{n-1} (1+t\lambda_i)^{\frac{1}{p}}}{\lambda_n^{1-\frac{1}{p}}t^{1-\frac{1}{p}}r^{1-\frac{n}{p}}}. 
 $$
The estimate \eqref{caplower} together with \eqref{estboxlow} gives us
 $$
   \cp_p^{1/p}(\varphi^{-1}(\hat{F});\varphi^{-1}(\hat{G})) \geq \frac{2^{\frac{n}{q}}r^{\frac{n}{q}-1}\prod\limits_{i=1}\limits^{n-1}(\lambda_i-\varepsilon)}{(t\lambda^2_n+\varepsilon)^{1-\frac{1}{q}}\prod\limits_{i=1}\limits^{n-1}(\lambda_i+t\lambda^2_i+\varepsilon)^{1-\frac{1}{q}}}.
 $$
Substituting these  estimates in the inequality \eqref{ringcap} we have
 $$
 \frac{2^{\frac{n}{q}}r^{\frac{n}{q}-1}\prod\limits_{i=1}\limits^{n-1}(\lambda_i-\varepsilon)}{(t\lambda^2_n+\varepsilon)^{1-\frac{1}{q}}\prod\limits_{i=1}\limits^{n-1}(\lambda_i+t\lambda^2_i+\varepsilon)^{1-\frac{1}{q}}}
 \leq C \widetilde\Phi_{p,q}(G)^{\frac{p-q}{pq}} \frac{2^{\frac{n}{p}}\prod\limits_{i=1}\limits^{n-1} (1+t\lambda_i)^{\frac{1}{p}}}{\lambda_n^{1-\frac{1}{p}}t^{1-\frac{1}{p}}r^{1-\frac{n}{p}}}.
 $$
 Let $\varepsilon \to 0$, then
 $$
 \frac{2^{\frac{n}{q}}r^{\frac{n}{q}-1}\prod\limits_{i=1}\limits^{n-1}\lambda_i}{(t\lambda^2_n)^{1-\frac{1}{q}}\prod\limits_{i=1}\limits^{n-1}(\lambda_i+t\lambda^2_i)^{1-\frac{1}{q}}}
 \leq C \widetilde\Phi_{p,q}(G)^{\frac{p-q}{pq}} \frac{2^{\frac{n}{p}}\prod\limits_{i=1}\limits^{n-1} (1+t\lambda_i)^{\frac{1}{p}}}{\lambda_n^{1-\frac{1}{p}}t^{1-\frac{1}{p}}r^{1-\frac{n}{p}}}
 $$
 Further, dividing both sides by $|G| = 2^n\lambda_n r^n t\prod\limits_{i=1}\limits^{n-1}(1+t\lambda_i)$ and tending $r \to 0$, we obtain
 $$
   \frac{\prod\limits_{i=1}\limits^{n-1}\lambda_i}{(t\lambda^2_n)^{1-\frac{1}{q}}\prod\limits_{i=1}\limits^{n-1}(\lambda_i+t\lambda^2_i)^{1-\frac{1}{q}}}
   \leq C (\widetilde\Phi'_{p,q}(y_0))^{\frac{p-q}{pq}} \frac{\lambda^{\frac{p-q}{pq}}_n t^{\frac{p-q}{pq}}\prod\limits_{i=1}\limits^{n-1} (1+t\lambda_i)^{\frac{p-q}{pq}}}{\lambda_n^{1-\frac{1}{p}}t^{1-\frac{1}{p}}}.
 $$
 Now multiplying both sides by $t^{1-\frac{1}{q}}$ and tending $t$ to $0$ we obtain
 $$
   \frac{\prod\limits_{i=1}\limits^{n-1}\lambda_i}{(\lambda^2_n)^{1-\frac{1}{q}}\prod\limits_{i=1}\limits^{n-1}\lambda_i^{1-\frac{1}{q}}}
   \leq C (\widetilde\Phi'_{p,q}(y_0))^{\frac{p-q}{pq}} \frac{\lambda^{\frac{p-q}{pq}}_n}{\lambda_n^{1-\frac{1}{p}}}.
 $$
 Finally we have
 $$
  \frac{\prod\limits_{i=1}\limits^{n}\lambda_i^{\frac{1}{q}}}{\lambda_n} \leq (\widetilde\Phi'_{p,q}(y_0))^{\frac{p-q}{pq}}.
 $$
 and the theorem is proved.
\end{proof}

\section{Capacitory characterizations of set functions}

In this section we give capacitory characterizations of the set function. Let $1<q<p<\infty$ and $\varphi:\Omega\to\widetilde{\Omega}$ be a homeomorphism. We define the set function on open bounded subsets $\widetilde A\subset\widetilde{\Omega}$ by the rule
\begin{equation}
\label{capfun}
\widetilde{\Psi}_{p,q}(\widetilde A) = \sup\limits_{(F; G) \subset \widetilde A} \left(\frac{\cp_{q}^{\frac{1}{q}}(\varphi^{-1}(F); \varphi^{-1}(G))}{\cp_{p}^{\frac{1}{p}}(F; G)}\right)^{\frac{pq}{p-q}},
\end{equation}
where the supremum is taken over all ring condensers $(F; G) \subset \widetilde A$ such that $\cp_{p}(F; G) \ne 0$.

Let us consider the variation of the set function $\widetilde{\Psi}_{p,q}$. Recall that the variation of a set function is a quasiadditive set function
$$
V(\widetilde{\Psi}_{p,q}, \widetilde\Omega) = \sup\limits_{\{\widetilde A_k\}} \sum\limits_{k=1}\limits^{\infty} \widetilde{\Psi}_{p,q}(\widetilde A_k),
$$
where the supremum is taken over all families of disjoint sets $\{\widetilde A_k\}_{k\in \mathbb{N}}$, $\widetilde A_k \subset \widetilde\Omega$.

\begin{prop}
\label{prop5}
Let $n-1 < q < p < \infty$, then the set function $\widetilde{\Psi}_{p,q}$ is a positive and non-decreasing function. 
\end{prop}

\begin{proof}
Let $\widetilde{A}_1\subset \widetilde{A}_2\subset \widetilde{\Omega}$. Then 
\begin{multline*}
\widetilde{\Psi}_{p,q}(\widetilde {A}_1) = \sup\limits_{(F_1; G_1) \subset \widetilde {A}_1} \left(\frac{\cp_{q}^{\frac{1}{q}}(\varphi^{-1}(F_1); \varphi^{-1}(G_1))}{\cp_{p}^{\frac{1}{p}}(F_1; G_1)}\right)^{\frac{pq}{p-q}}\\
=\sup\limits_{(F_1; G_1) \subset \widetilde {A}_2} \left(\frac{\cp_{q}^{\frac{1}{q}}(\varphi^{-1}(F_1); \varphi^{-1}(G_1))}{\cp_{p}^{\frac{1}{p}}(F_1; G_1)}\right)^{\frac{pq}{p-q}} \\
\leq \sup\limits_{(F_2; G_2) \subset \widetilde {A}_2} \left(\frac{\cp_{q}^{\frac{1}{q}}(\varphi^{-1}(F_2); \varphi^{-1}(G_2))}{\cp_{p}^{\frac{1}{p}}(F_2; G_2)}\right)^{\frac{pq}{p-q}} = \widetilde{\Psi}_{p,q}(\widetilde {A}_2  ),
\end{multline*}
because any condenser $(F_1; G_1) \subset \widetilde {A}_2$ coincides with some condenser $(F_2; G_2) \subset \widetilde {A}_2$.

Now we prove positivity of the set function $\widetilde{\Psi}_{p,q}(\widetilde A)$. Fix an arbitrary point $x \in \widetilde A$ and choose $r > 0$ such that ring condenser $(\overline B(x, r); B(x, 2r))$ is a subset of $\widetilde A$. Then, by \cite{GResh},
$$
\cp_p(\overline B(x, r); B(x, 2r))\geq c_1>0.
$$
The preimage $\varphi^{-1}(\overline B(x, r))$ is a continuum set. Then because in the case $q>n-1$ the capacity of a continuum set is positive \cite{GResh}, we have
$$
\cp_p(\varphi^{-1}(\overline B(x, r)); \varphi^{-1}(B(x, 2r)))\geq c_2>0. 
$$
Hence $\widetilde{\Psi}_{p,q}(\widetilde A)>0$ on arbitrary open sets $\widetilde A \subset \mathbb{R}^n$.

\end{proof}

\begin{rem}
In the case $n=2$ Proposition~\ref{prop5} is correct for all $1\leq q<p<\infty$. 
\end{rem}

In the following assertion we give connections between set functions $\widetilde{\Phi}_{p,q}$ defined by the norms of composition operators ans set functions $\widetilde{\Psi}_{p,q}$ defined by capacity. 

\begin{thm}\label{setcap}
Fix a constant $M<\infty$ and let $\varphi: \Omega \to \widetilde\Omega$ be a homeomorphism. Then the following conditions are equivalent:
\begin{enumerate}
\item There exists the bounded monotone countable-quasiadditive set function
$\widetilde{\Phi}_{p,q}$ defined on open subsets of $\widetilde{\Omega}$
such that for every ring condenser 
$(F; \widetilde A)\subset \widetilde{\Omega}$
the inequality
$$
\cp_{q}^{1/q}(\varphi^{-1}(F);\varphi^{-1}(\widetilde A))
\leq \widetilde{\Phi}_{p,q}(\widetilde A)^{\frac{p-q}{pq}}
\cp_{p}^{1/p}(F; \widetilde A)
$$
holds and $ \widetilde{\Phi}_{p,q}(\widetilde\Omega) \leq M$;
\item The set function $\widetilde{\Psi}_{p,q}$ defined by (\ref{capfun}) has the bounded variation $V(\widetilde{\Psi}_{p,q}, \widetilde\Omega) \leq M$;
\item For every family of disjoint ring condensers $\{(F_k; \widetilde A_k)\}_{k\in \mathbb{N}}$ in $\widetilde\Omega$ with $\cp_{p}^{\frac{1}{p}}(F_k; \widetilde A_k) \ne 0$, the following inequality
\begin{equation}\label{cap_sum}
\sum\limits_{k=1}\limits^{\infty} \left( \frac{\cp_{q}^{\frac{1}{q}}(\varphi^{-1}(F_k); \varphi^{-1}(\widetilde A_k))}{\cp_{p}^{\frac{1}{p}}(F_k; \widetilde A_k)} \right)^{\frac{pq}{p-q}} \leq M
\end{equation}
holds.
\end{enumerate}
\end{thm}

\begin{proof}
On the first step we prove the equivalence of the second and the third conditions of the theorem. Let $V(\widetilde{\Psi}_{p,q}, \widetilde\Omega) \leq M$ then, by the  definitions of set functions $\widetilde{\Psi}$ and $V(\widetilde{\Psi}_{p,q})$, we have
$$
\sup\limits_{\{\widetilde A_k\}} \sum\limits_{k=1}\limits^{\infty} \sup\limits_{(F_k; G_k) \subset \widetilde A_k} \left( \frac{\cp_{q}^{\frac{1}{q}}(\varphi^{-1}(F_k); \varphi^{-1}(G_k))}{\cp_{p}^{\frac{1}{p}}(F_k; G_k)} \right)^{\frac{pq}{p-q}} \leq M,
$$
where the first supremum is taken over all families of disjoint sets $\widetilde A_k$, and the second is taken over all ring condensers in $\widetilde A_k$ with non-zero capacity. Hence, for an arbitrary family $\{\widetilde A_k\}_{k \in \mathbb N}$ and condensers $(F_k; \widetilde A_k)$, the sum \eqref{cap_sum} is bounded by the same constant $M$.

Now we assume that the condition (3) is hold. Then for an arbitrary family of disjoint sets $\{(F_k; \widetilde A_k)\}_{k \in \mathbb N}$ in $\widetilde\Omega$ and an arbitrary number $\varepsilon > 0$, we choose ring condensers $(F_k; G_k)$ such that
$$
\widetilde{\Psi}_{p,q}(\widetilde A_k) < \left( \frac{\cp_{q}^{\frac{1}{q}}(\varphi^{-1}(F_k); \varphi^{-1}(G_k))}{\cp_{p}^{\frac{1}{p}}(F_k; G_k)} \right)^{\frac{pq}{p-q}} + \frac{\varepsilon}{2^k}.
$$
Then we obtain $\sum\limits_{k=1}\limits^{\infty} \widetilde{\Psi}_{p,q}(\widetilde A_k) < M + \varepsilon$. So, because $\varepsilon > 0$ is an arbitrary number we have $V(\widetilde{\Psi}_{p,q}, \widetilde\Omega) \leq M$ and hence the condition~(2) is hold.

Next we prove that the conditions (1) and (2) are equivalent. Let the condition~(1) holds, then by the capacitary inequality 
$$
\cp_{q}^{1/q}(\varphi^{-1}(F);\varphi^{-1}(\widetilde A))
\leq \widetilde{\Phi}_{p,q}(\widetilde A)^{\frac{p-q}{pq}}
\cp_{p}^{1/p}(F; \widetilde A)
$$
we have $\widetilde{\Psi}_{p,q}(\widetilde A) \leq \sup \widetilde{\Phi}_{p,q}(G)$, where the supremum is taken over all $G \subset \widetilde A$. Hence, by monotonicity of the set function $\widetilde{\Phi}_{p,q}$ we obtain $\widetilde{\Psi}_{p,q}(\widetilde A) \leq \widetilde{\Phi}_{p,q}(\widetilde A)$ for every $\widetilde A \subset \widetilde{\Omega}$. By the definition of the variation
$$
V(\widetilde{\Psi}_{p,q}, \widetilde\Omega) = \sup\limits_{\{\widetilde A_k\}} \sum\limits_{k=1}\limits^{\infty} \widetilde{\Psi}_{p,q}(\widetilde A_k) \leq \sup\limits_{\{\widetilde A_k\}} \sum\limits_{k=1}\limits^{\infty} \widetilde{\Phi}_{p,q}(\widetilde A_k).
$$
The set function $\widetilde{\Phi}_{p,q}$ is countable-quasiadditive and monotone, so it implies 
$$
V(\widetilde{\Psi}_{p,q}, \widetilde\Omega) \leq \sup \widetilde{\Phi}_{p,q}(\bigcup\limits_{k=1}\limits^{\infty} \widetilde A_k) \leq \widetilde{\Phi}_{p,q}(\widetilde\Omega)
$$
and the condition~(2) holds. 

Suppose that the condition~(2) holds, then by the definition of set functions $\widetilde{\Psi}_{p,q}$ and $V(\widetilde{\Psi}_{p,q})$ we obtain that the inequality
$$
\cp_{q}^{1/q}(\varphi^{-1}(F);\varphi^{-1}(\widetilde A))
\leq V(\widetilde{\Psi}_{p,q}, \widetilde A)
\cp_{p}^{1/p}(F; \widetilde A).
$$
holds for any condenser $(F, \widetilde A)$ with $\cp_p(F, \widetilde A) \ne 0$. The variation $V(\widetilde{\Psi}_{p,q})$ is a quasiadditive set function. 

Now we prove that the above inequality in the case $\cp_p(F; \widetilde A) = 0$ also holds also. Consider an arbitrary condenser $(F; \widetilde A)$ with $\cp_p(F; \widetilde A) = 0$. There exists a decreasing sequence of compact sets $\{F_k\}_{k=1}^\infty$, $F_k \subset \widetilde A$, $F = \bigcap_{k=1}^\infty F_k$, with $\cp_p(F_k; \widetilde A) > 0$ (for example, $F_k = \{x: \dist(x, F) \leq 1/k\}$). Then
$$
\cp_{q}^{1/q}(\varphi^{-1}(F_k);\varphi^{-1}(\widetilde A))
\leq V(\widetilde{\Psi}_{p,q}, \widetilde A)
\cp_{p}^{1/p}(F_k; \widetilde A).
$$

By the continuity of the capacity we can pass to the limit and obtain the above inequality for condensers with zero capacity.
\end{proof}

Note, that the similar statement was proven in the particular case $p=n$ in \cite{K86}.

\section{Capacitary metrics}

In this section we introduce a capacitary metric of the hyperbolic type and prove, that $(p,q)$-quasiconformal mappings are Lipschitz mappings in correspondence capacitary metrics.

Let $n-1<p\leq n$ ($1\leq p\leq 2$ if $n=2$) and let a domain $\Omega\subset\mathbb R^n$ have a boundary of non-zero $p$-capacity.

Let us define a capacitary metric 
$$
d_p(x,y) = \inf\limits_{\gamma} \cp^{\frac{1}{p}}_p(\gamma; \Omega),
$$
where the infimum is taken over all continuous curves $\gamma$, joining points $x$ and $y$ in $\Omega$. 
In the case $p=n$ metrics of such type was considered in \cite{FMV91} in a connection with quasiconformal mappings.

\begin{thm} Let $n-1<p\leq n$ ($1\leq p\leq 2$ if $n=2$) and $\cp_p(\partial\Omega)>0$ then 
$$
d_p(x,y) = \inf\limits_{\gamma} \cp^{\frac{1}{p}}_p(\gamma; \Omega)
$$ is a metric.
\end{thm}

\begin{proof}

The restriction on $p$ arises because for $p>n$ only empty set has zero capacity and for $p \leq n-1$ even a continuum has zero capacity. If boundary $\partial \Omega$ has zero capacity, then, by the definition of $\cp^{\frac{1}{p}}_p(\gamma; \Omega)$ will be equal to zero for any continuous curve $\gamma$.

\noindent
1. {\bf Identity}: 

If $x = y$ then the infimum is reached on one-point set, which has capacity zero, hence, $d_p(x,y)=0$. Now, let $d_p(x,y)=0$. 
Suppose that $|x-y|>0$. Then by the capacity estimate \cite{K86} for any continuous curve $\gamma\subset\Omega$ which joint points $x$ and $y$ we have
$$
\cp_p^{n-1}(\gamma;\Omega)\geq c(n,p)\frac{|x-y|^p}{|\Omega|^{p-n+1}}>0.
$$
Contradiction.

\noindent
2. {\bf Symmetry}: Follows from the definition $d_p(x,y)$:
$$
d_p(x,y) = \inf\limits_{\gamma} \cp^{\frac{1}{p}}_p(\gamma; \Omega)=d_p(y,x)
$$

\noindent
2. {\bf Triangle inequality}: Follows from the triangle inequality for $L^1_p$-seminorms in the definition of the capacity. 
\end{proof}

By using Theorem~\ref{cap-pp} and Theorem~\ref{cap-pq} we obtain the following statement.
\begin{thm}
Let $\varphi: \Omega \to \widetilde\Omega$ be a $(p,q)$-quasiconformal mapping, $n-1< q \leq p \leq n$ ($1\leq q\leq p\leq 2$ if $n=2$). Then, for any two points $x,y \in \widetilde\Omega$ the following inequality 
$$
d_q(\varphi^{-1}(x), \varphi^{-1}(y)) \leq K_{p,q}(\varphi;\Omega) d_p(x,y),
$$
holds.
\end{thm}
\begin{proof}
Because $\varphi: \Omega \to \widetilde\Omega$ is a $(p,q)$-quasiconformal mapping, then for every ring condenser  $F\subset \widetilde{\Omega}$
the inequality
$$
\cp_{q}^{1/q}(\varphi^{-1}(F);\Omega)
\leq  K_{p,q}(\varphi;\Omega)
\cp_{p}^{1/p}(F;\widetilde{\Omega})
$$
holds.
Let $\widetilde\gamma \subset \widetilde\Omega$ be a continuous curve which join $x$ and $y$ in $\widetilde{\Omega}$. Because $\varphi$ is a homeomorphism, then $\varphi^{-1}(\widetilde\gamma)$ is a  continuous curve which join $\varphi^{-1}(x)$ and $\varphi^{-1}(y)$ in ${\Omega}$.
Hence we obtain the following inequalities:
$$
d_q(\varphi^{-1}(x), \varphi^{-1}(y)) \leq \cp^{\frac{1}{q}}_q(\varphi^{-1}(\widetilde\gamma); \Omega) \\
\leq K_{p,q}(\varphi;\Omega) \cp^{\frac{1}{p}}_p(\widetilde\gamma; \widetilde\Omega).
$$
Now since $\widetilde\gamma$ is an arbitrary continuous curve which join $x$ and $y$, then, taking the infimum over all such $\widetilde\gamma$, we obtain 

$$
 d_q(\varphi^{-1}(x), \varphi^{-1}(y)) \leq K_{p,q}(\varphi;\Omega) \inf \cp^{\frac{1}{p}}_p(\widetilde\gamma; \widetilde\Omega)=K_{p,q}(\varphi;\Omega)~ d_p(x,y).
$$

\end{proof}

Using the composition duality theorem (Theorem \ref{CompThD}) for composition operators, we immediately obtain the following result:

\begin{thm}
Let $\varphi: \Omega \to \widetilde\Omega$ be a $(p,q)$-quasiconformal mapping, $n < q \leq p < \infty$. Then, for any two points $x,y \in \Omega$ the following inequality 
$$
d_{p'}(\varphi(x), \varphi(y)) \leq K_{q',p'}(\varphi^{-1};\widetilde{\Omega})~d_{q'}(x, y),\,\, p' = \frac{p}{p-n+1}, 
q' = \frac{q}{q-n+1}
$$
holds.
\end{thm}

\vskip 0.3cm

Alexander Menovschikov; Department of Mathematics, University of Hradec Kr\'alov\'e, Rokitansk\'eho 62, 500 03 Hradec Kr\'alov\'e, Czech Republic.
 
\emph{E-mail address:} \email{alexander.menovschikov@uhk.cz} \\

Alexander Ukhlov; Department of Mathematics, Ben-Gurion University of the Negev, P.O.Box 653, Beer Sheva, 8410501, Israel 
							
\emph{E-mail address:} \email{ukhlov@math.bgu.ac.il}

\end{document}